\documentclass[leqno,12pt]{amsart} 

\usepackage{tikz}

\usepackage{mathrsfs}

\usepackage{amssymb}
\theoremstyle{plain} 
\newtheorem{theorem}{\indent\sc Theorem}[section]
\newtheorem{lemma}[theorem]{\indent\sc Lemma}

\theoremstyle{definition} 

\begin{document}

\title[Disproofs of two conjectures concerning nondeficient numbers]{Disproofs of two conjectures concerning nondeficient numbers} 

\author[J. M.\ Campbell]{John M.\ Campbell} 

\subjclass[2020]{ 
Primary 11A25.
}
%
\keywords{ 
 Nondeficient number, abundant number, perfect number, divisor, primorial.
}
\address{
 Department of Mathematics and Statistics \endgraf
 Dalhousie University \endgraf
 Halifax, Nova Scotia, Canada
}
\email{jh241966@dal.ca}

\maketitle

\begin{abstract}
A positive integer $n$ is said to be \emph{nondeficient} if $\sigma(n) \geq 2n$. Letting the positive divisors of a positive integer 
 $n$ be written as $1 = d_0 < d_1 < \cdots < d_k < d_{k+1} = n$, and letting $\mathcal{S}$ denote a set of integers, if there exist values 
 $\lambda_j \in \mathcal{S}$ such that $1 + \sum_{j=1}^{k} \lambda_j d_j = n$, then $n$ is said to be an \emph{$\mathcal{S}$-perfect 
 number}. Ross, in 2024, introduced the study of $\mathcal{S}$-perfect numbers, and concluded with two conjectures that each concern 
 both $\{ -1, 1 \}$-perfect numbers and nondeficient numbers. We disprove both of these conjectures.
\end{abstract}

\section{Introduction}
 Let $\sigma_x(n) = \sum_{d \mid n} d^{x}$, and write $\sigma_1(n) = \sigma(n)$. A \emph{perfect number} is a positive integer $n$ that 
 satisfies $\sigma(n) = 2n$. Similarly, a \emph{nondeficient number} $n$ is a positive integer that satisfies $\sigma(n) \geq 2n$. The study 
 of perfect numbers and of variants and 
 generalizations of perfect numbers provides an active area within number theory. The purpose of this paper is to disprove conjectures 
 that were introduced by Ross in 2024 \cite{Ross2024} and that concern nondeficient numbers and a variant of perfect numbers. 

 There is a long history about the study of perfect numbers, again with reference to the work of Ross \cite{Ross2024} along with works 
 cited therein, as in Hassler's paper on the early history surrounding perfect numbers \cite{Hassler2023}. To this day, it is not even known 
 whether or not there are infinitely many perfect numbers, and, indeed, this provides a well-known open problem in number 
 theory. Similarly, it is currently not known whether or not there are any odd perfect numbers, and, again, this provides a well-known 
 open problem in number theory. The foregoing considerations motivate the development of techniques in the study of perfect numbers 
 and variants/generalizations of perfect numbers. 

 For a positive integer $n$, let the positive divisors of $n$ be written so that 
\begin{equation}\label{orderdiv}
 1 = d_0 < d_1 < \cdots < d_k < d_{k+1} = n. 
\end{equation}
 Let $\mathcal{S}$ be a set of integers. According to Ross \cite{Ross2024}, the integer $n$ is said to be \emph{$\mathcal{S}$-perfect} if 
 there are integers $\lambda_1$, $\lambda_2$, $\ldots$, $\lambda_k \in \mathcal{S}$ satisfying $1 + \sum_{j=1}^{k} \lambda_j d_j = n$. 
 Ross introduced and proved a characterization for all even $\{ -1, 1 \}$-perfect numbers, and concluded with two conjectures that 
 both concern both $\{ -1, 1 \}$-perfect numbers and nondeficient numbers. Building on our extensive interactions with GPT-5.5 Pro, 
 we disprove both of Ross's conjectures. 

\subsection{Ross's conjectures}\label{introconjectures}
 A positive integer $n$ is said to be \emph{abundant} if $\sigma(n) > 2 n$. Some authors, including Ross~\cite{Ross2024}, use this same 
 term in reference to numbers $n$ satisfying $\sigma(n) \geq 2n$, as opposed to $\sigma(n) > 2n$. For clarity, we make use of the 
 standard definitions for abundant, nondeficient, and perfect numbers. We also restrict the notion of a number being $\{ -1, 1 \}$-perfect 
 to positive integers, as suggested above (and this agrees with the below conjectures from Ross). 

 As observed by Ross, every $\{ -1, 1 \}$-perfect number is nondeficient, but it is not the case that every nondeficient number is $\{ -1, 
 1 \}$-perfect \cite{Ross2024}. This led Ross to conjecture that the $\{ -1, 1 \}$-perfect numbers have a density equal to the density of the 
 nondeficient numbers (cf.\ \cite[Conjecture 16]{Ross2024}). 
 We disprove this conjecture in Section \ref{firstdisproof}. 
 This is motivated by the problem of evaluating the density of abundant numbers, 
 (reviewed in Section \ref{secBack} below), and this 
 is also motivated by the recent work of Zelinsky on nondeficient numbers \cite{Zelinsky2026}. 

 Ross also observed that if an integer $n$ is odd and nondeficient, then $n^2$ is also odd and nondeficient, but $n^2$ cannot be $\{ -1, 
 1 \}$-perfect \cite{Ross2024}. This led Ross to conjecture (see Conjecture 17) that every nonsquare odd nondeficient 
 number is $\{ -1, 1 \}$-perfect. 
 We disprove this conjecture in Section \ref{seconddisproof}. 

\subsection{Background}\label{secBack}
 Since the first Ross conjecture given above in Section \ref{introconjectures} concerns the natural density of nondeficient numbers, and 
 since there has been much in the way of research related to the natural density of abundant numbers, we emphasize that these densities 
 are equal. This follows from the density of perfect numbers being $0$, 
 which is well-known, and which can be shown to follow from 
 a result due to Davenport \cite{Davenport1933} giving that the limit 
\begin{equation}\label{defineDu}
 D(u)=\lim_{x\to\infty}\frac1x
 \#\{n\le x:n/\sigma(n)\le u\}
\end{equation}
 exists and that the function of $u \in [0, 1]$ defined in \eqref{defineDu} is continuous. A simplified proof of this result was  subsequently 
 given by  Erd\H{o}s \cite{Erdos1934}. 

 There have been many research efforts concerning the problem of estimating 
 the density $D\big( \frac{1}{2} \big)$ of abundant numbers. 
 In this direction, a notable estimate is given by the bounds 
\begin{equation}\label{Behrendbound}
 0.241 < D\left( \frac{1}{2} \right) < 0.314 
\end{equation}
 given by Behrend in 1933 \cite{Behrend1933}, 
 with reference to subsequent refinements of 
 \eqref{Behrendbound} due to Wall 
 et al.\ \cite{Wall1972,Wall1970,WallCrewsJohnson1972}, due 
 to Del\'eglise \cite{Deleglise1998}, and due to Kobayashi \cite{Kobayashi2010}. 
 Since the density of nondeficient numbers is $D\big( \frac{1}{2} \big)$, 
 and again since the first Ross conjecture \cite{Ross2024} 
 reproducced in Section \ref{introconjectures}
 concerns this same density, 
 this motivates the techniques we apply 
 in Section \ref{firstdisproof} to disprove the first Ross conjecture. 

 In addition to Ross's definition of $\mathcal{S}$-perfect numbers \cite{Ross2024},   many further extensions and variants of perfect numbers  
  have been introduced 
 and extensively studied. As indicated above, this forms a major area within number theory. 
 In turn, this motivates our techniques applied in this paper
 concerning the distributions of $\{ -1, 1 \}$-perfect numbers and nondeficient numbers. 
 Well-known extensions or variants of perfect numbers include sociable numbers
 (with the order-1 sociable numbers being perfect numbers), amicable numbers (which are order-2 sociable numbers), 
 multiply perfect numbers, superperfect numbers, and hyperperfect numbers. 
 Further examples of extensions/variants of perfect numbers include 
 pseudoperfect numbers \cite{Sierpinski1965}, 
 prime-perfect numbers \cite{PollackPomerance2012}, 
 $s$-near perfect numbers \cite{PollackShevelev2012}, and 
 within-perfect numbers~\cite{KwanMiller2024}. 

\section{Disproof of Ross's first conjecture}\label{firstdisproof}
 Following Ross \cite{Ross2024}, we write $P(\mathcal{S})$ for the set of positive $\mathcal{S}$-perfect numbers of the first kind, omitting 
 the brackets associated with the notation for a given finite subset
 $\mathcal{S}$ of $\mathbb{Z}$. We then let 
 $ \mathcal{P} = P(-1,1)$. 
 Ross observed that the sequence of nondeficient numbers that are not $\{ -1, 1 \}$-perfect 
 begins with $ 18$, $20$, $36$, $\ldots$. 
 Informally, our strategy to disprove Ross's first conjecture
 may be thought of as being based on extending the first example of a nondeficient number that is
 not $\{ -1, 1 \}$-perfect. 

 For a given set $\mathcal{S}$ of integers, we adopt the notational conventions such that
\begin{equation}\label{notationlimsup} 
 \overline d(\mathcal S)=\limsup_{x\to\infty}
 \frac{\#\{n\le x:n\in\mathcal S\}}{x},
 \qquad
 \underline d(\mathcal S)=\liminf_{x\to\infty}
 \frac{\#\{n\le x:n\in\mathcal S\}}{x}.
\end{equation} 
 If the two values in \eqref{notationlimsup}
 are equal, then we may write $d(\mathcal S)$ in place of their common value. 
 
 We write 
\begin{equation}\label{definecalA}
 \mathcal{A} =\{n\ge 1:\sigma(n)\ge 2n\}. 
\end{equation}
 The main goal of this section is to prove Theorem \ref{thm:main} below, 
 giving us that 
 there is a set $ \mathcal{E} \subseteq \mathcal{A} \setminus \mathcal{P} $ satisfying $\underline d(\mathcal{E})>0$, i.e., so that 
 $ \overline d(\mathcal{P})<d( \mathcal{A} ) $ holds, thereby 
 disproving Ross's density conjecture \cite[Conjecture 16]{Ross2024} concerning $\mathcal{P}$. 

 We begin with a reformulation, as below, of the definition of a $\{ -1, 1 \}$-perfect number. 

\begin{lemma}\label{lem:criterion}
 For an integer $n$ exceeding $1$, let the positive divisors of $n$
 be written as in \eqref{orderdiv}. 
 Then $n \in \mathcal{P}$ if and only if
$ \frac{\sigma(n)-2n}{2} $
is a sum of distinct divisors chosen from $d_1, d_2, \ldots,d_k$, allowing for the possibility of an empty sum. 
\end{lemma}

\begin{proof}
 Since $ 1+\sum_{j=1}^k d_j = \sigma(n)-n$, by switching the sign of $d_j$ for a given index $j \in \{ 1, 2, \ldots, k \}$, we see that: The 
 existence of an expansion of the form $ n=1+\sum_{j=1}^k c_jd_j 
$ for $c_j\in\{-1,1\}$
 is equivalent to the existence of a (possibly empty) subset $T\subseteq\{1,\ldots,k\}$ such that 
 $ \sigma(n) - 2n = 2 \sum_{j \in T} d_j$, as desired. 
\end{proof}

 Observe that Lemma \ref{lem:criterion} confirms that each positive $\{-1,1\}$-perfect number is nondeficient (as observed by Ross 
 \cite{Ross2024}), since a (possibly empty) sum of divisors
 of the form given in Lemma \ref{lem:criterion}
 gives us that $\sigma(n) - 2n \geq 0$. 

 We make use of standard notation for arithmetic functions, letting $\omega(n)$ denote the number of distinct prime factors of a 
 positive integer $n$, 
 letting $\varphi(n)$ denote the Euler totient function giving the number
 of integers $m \in [1, n]$ relatively prime with $n$, 
 letting $\pi(x)$ denote the number of primes not exceeding $x$, 
 and letting 
 $p_n\# = \prod_{i=1}^{n} p_{i}$ denote the $n^{\text{th}}$ primorial. 
 For an integer $n$ exceeding $1$, by writing 
 $n = \prod_{i=1}^{\omega(n)} p_{a_i}^{b_i}$, 
 the relation 
\begin{equation}\label{basicsigma} 
 \sigma_x(n) = \prod_{i=1}^{\omega(n)} \left( 1 + p_{a_{i}}^{x} + p_{a_{i}}^{2x} + \cdots + p_{a_{i}}^{b_i x} \right) 
\end{equation}
 provides a basic property 
 of the arithmetic function $\sigma_x(n)$ \cite[\S16.7]{HardyWright2008} and is required for our purposes. 

 We also require a standard consequence of Markov's inequality, 
 providing what may be seen as a ``non-probabilistic'' version of Markov's inequality. 
 Explicitly, if $E$ is a finite set, then, for a function $F\colon E \to [0, \infty)$ and a positive value $\eta$, 
 we have that 
\begin{equation}\label{Markovinequality} 
 \eta \, \left| \{ u \in E : F(u) \geq \eta \} \right| \leq \sum_{u \in E} F(u). 
\end{equation}

\begin{lemma}\label{lem:small-ratio}
 For every $\varepsilon>0$, the set of positive integers $m$ such that $ (m,6)=1$ and $ \frac{\sigma(m)}{m}<1+\varepsilon$ 
has positive lower asymptotic density.
\end{lemma}

\begin{proof}
 Let $y$ be a real number such that $y \geq 3$. We proceed to define
\begin{equation}\label{defineRy}
 \mathcal{R}_{y} = \{ m \geq 1 : (m, p_{\pi(y)}\# ) = 1 \}. 
\end{equation}
 By rewriting the right-hand side of \eqref{defineRy}
 as a union of residue classes modulo $p_{\pi(y)}\#$, 
 we find that it has a natural density equal to 
\begin{equation}\label{rhoy} 
 \rho_{y} = \prod_{p \leq y} \left( 1 - \frac{1}{p} \right) = \frac{\varphi(p_{\pi(y)}\#) }{p_{\pi(y)}\# }, 
\end{equation}
 i.e., so that 
\begin{equation}\label{densityRy}
 \left| \{ m \leq x : m \in \mathcal{R}_{y} \} \right| 
 = \rho_{y} x + O_y(1).
\end{equation}
 Since $y \geq 3$, we have that $2 \mid p_{\pi(y)}\#$ and $3 \mid p_{\pi(y)}\#$. 
 Consequently, the implication 
\begin{equation}\label{primetosix}
 m \in \mathcal{R}_{y} \Longrightarrow (m, 6) = 1 
\end{equation}
 holds. Now, we define
\begin{equation}\label{defineGy}
 G_{y}(m) = \sum_{ \substack{p \mid m \\ p > y} } \frac{1}{p-1}. 
\end{equation}
 If $m \in \mathcal{R}_{y}$, then each prime $p \leq y$ does not divide $m$. 
 So, if $m \in \mathcal{R}_{y}$, then, since each prime divisor of $m$ is strictly greater than $y$, 
 the right-hand side of \eqref{defineGy} may be rewritten so that 
 $G_y(m) = \sum_{p \mid m} \frac{1}{p-1}$. 

 Now, by writing a positive integer $m$
 so that $m = \prod_{p^b \Vert m} p^b$, this together with the relation in \eqref{basicsigma}
 give us that 
\begin{align*}
 \frac{\sigma(m)}{m} 
 & = \prod_{p^a \Vert m} \left( 1 + \frac{1}{p} + \cdots + \frac{1}{p^a} \right) \\ 
 & \leq \prod_{p \mid m} \left( 1 - \frac{1}{p} \right)^{-1}. 
\end{align*}
 If $u \in (0, 1)$, then
 $-\log(1-u) \leq \frac{u}{1-u}$. 
 This together with the relation 
 $\log \frac{\sigma(m)}{m} \leq \sum_{p \mid m} - \log\left( 1 - \frac{1}{p} \right)$
 give us that 
\begin{equation}\label{convertGy} 
 \log \frac{\sigma(m)}{m} \leq \sum_{p \mid m} \frac{1}{p-1}. 
\end{equation}
 So, from \eqref{convertGy} along with how the right-hand side of 
 \eqref{defineGy} reduces for $m \in \mathcal{R}_{y}$, we find that 
\begin{equation}\label{ifminRy} 
 m \in \mathcal{R}_{y} \Longrightarrow \log \frac{\sigma(m)}{m} \leq G_y(m).
\end{equation}

 By defining 
\begin{equation}\label{defineSy}
 S_{y}(x) = \sum_{\substack{m \leq x \\ m \in \mathcal{R}_{y} }} G_{y}(m), 
\end{equation}
 we proceed to rewrite the right-hand side of \eqref{defineSy} so that 
\begin{align}
\begin{split}
 S_{y}(x) 
 & = \sum_{\substack{m \leq x \\ m \in \mathcal{R}_{y} }} \sum_{ \substack{p \mid m \\ p > y} } \frac{1}{p-1} \\ 
 & = \sum_{p \in (y, x]} \frac{1}{p-1} \left| \{ m \leq x : m \in \mathcal{R}_{y}, p \mid m \} \right|.
\end{split}\label{splitdoublesum}
\end{align}
 We proceed to fix a prime $p > y$. 
 From the inequality $p > y$, we have that $p \nmid p_{\pi(y)}\#$. 
 Now, let $m$ be an integer multiple of $p$, writing $m = p u$. 
 With this assumption, we obtain the equivalences
\begin{equation}\label{equivalences} 
 m \in \mathcal{R}_{y} \Longleftrightarrow (pu, p_{\pi(y)}\#) = 1 \Longleftrightarrow (u, p_{\pi(y)}\#) = 1. 
\end{equation}
 In turn, the equivalences in \eqref{equivalences} give us that 
\begin{equation}\label{equalcard} 
 \left| \{ m \leq x : m \in \mathcal{R}_{y}, p \mid m \} \right| 
 = \left| \left\{ u \leq \frac{x}{p} : (u, p_{\pi(y)}\#) = 1 \right\} \right|. 
\end{equation}

 For real $X \geq 0$, we define 
\begin{equation}\label{defineNy} 
 \mathcal{N}_{y}(X) = \left| \{ u \leq X : (u, p_{\pi(y)}\#) = 1 \} \right|. 
\end{equation} 
 Now, write 
\begin{equation}\label{floorX}
 \lfloor X \rfloor = a \, p_{\pi(y)}\# + b
\end{equation}
 for $a \geq 0$ and for $b$ such that 
\begin{equation}\label{boundb}
 0 \leq b < p_{\pi(y)}\#, 
\end{equation}
 and define
\begin{equation}\label{mathcalSy}
 \mathcal{T}_{y}(b) = \left| \{ u \in [1, b] : (u, p_{\pi(y)}\#) = 1 \} \right|. 
\end{equation}
 For a positive integer $z$, the indicator function determining whether or not a positive integer is coprime with $z$ is periodic, with 
 period $z$. For the case of $z = p_{\pi(y)}\#$, each occurrence of the period contains $\varphi(p_{\pi(y)}\#)$ integers coprime 
 to $p_{\pi(y)}\#$. So, by exploiting this periodicity phenomenon, 
 and by counting integers $u \leq X$ coprime with $p_{\pi(y)}\#$, 
 we may rewrite the right-hand side of 
 \eqref{defineNy} so that 
\begin{equation}\label{Nyphi}
 \mathcal{N}_{y}(X) = a \, \varphi(p_{\pi(y)}\#) + \mathcal{T}_{y}(b). 
\end{equation}
 From \eqref{rhoy} and \eqref{Nyphi} together, we find that 
\begin{equation}\label{Nyrho}
 \mathcal{N}_{y}(X) = \rho_{y} a \, p_{\pi(y)}\# + \mathcal{T}_{y}(b). 
\end{equation}
 In a similar spirit, the relations in \eqref{floorX} and \eqref{Nyrho} together give us that 
\begin{equation}\label{needsbounds} 
 \mathcal{N}_{y}(X) - \rho_{y} X = \mathcal{T}_{y}(b) - \rho_{y} b + \rho_{y} \big( \lfloor X \rfloor - X \big). 
\end{equation}
 From the bounds for $b$ in \eqref{boundb} together with 
 the right-hand side of \eqref{mathcalSy}, we see that 
\begin{equation}\label{boundSy} 
 0 \leq \mathcal{T}_{y}(b) < p_{\pi(y)}\#. 
\end{equation}
 Recalling the density $\rho_y$ shown in 
 \eqref{rhoy}, we find that 
 $0 \leq \rho_y b \leq b$, so that the bounds for $b$ in \eqref{boundb} give us that 
\begin{equation}\label{boundrhoy}
 0\leq \rho_y b < p_{\pi(y)}\#.
\end{equation}
 So, from \eqref{needsbounds} together with the bounds in both 
 \eqref{boundSy} and \eqref{boundrhoy}, we find that 
\begin{equation}\label{reducetoO1}
 \mathcal{N}_{y}(X) = \rho_y X + O(p_{\pi(y)}\#). 
\end{equation}
 With the understanding that $y$ is fixed, the relation in \eqref{reducetoO1} is equivalent to 
\begin{equation}\label{rhoyXerror}
 \left| \{ u \leq X : (u, p_{\pi(y)}\#) = 1 \} \right| = \rho_y X + O_{y}(1). 
\end{equation}
 By setting $X = \frac{x}{p}$, the relations in \eqref{equalcard} 
 and \eqref{rhoyXerror} together yield 
\begin{equation}\label{placeinsummand}
 \left| \{ m \leq x : m \in \mathcal{R}_{y}, p \mid m \} \right| 
 = \rho_y \frac{x}{p} + O_{y}(1). 
\end{equation}
 From \eqref{splitdoublesum} and \eqref{placeinsummand} together, we find that 
\begin{equation}\label{rewriteloglog}
 S_{y}(x) = \rho_y x \sum_{p \in (y, x]} \frac{1}{ p (p-1)} + 
 \left( \sum_{p \in (y, x]} \frac{1}{p-1} \right) O_{y}(1). 
\end{equation}
 From \eqref{rewriteloglog}, we write 
\begin{equation}\label{useconvergence}
 \frac{S_{y}(x)}{x} 
 = \rho_y \sum_{p \in (y, x]} \frac{1}{ p (p-1)} + 
 O_{y}\left( \frac{\log \log x}{x} \right). 
\end{equation}
 The convergence of $\sum_{p} \frac{1}{p(p-1)}$ together with the relation in 
 \eqref{useconvergence} give us that 
\begin{equation}\label{limfracSy} 
 \lim_{x \to \infty} \frac{S_{y}(x)}{x} = \rho_y \sum_{p > y} \frac{1}{p(p-1)}. 
\end{equation}
 Observe that $ \sum_{p > y} \frac{1}{p(p-1)} \to 0$ as $y \to \infty$. 
 So, for arbitrary $\varepsilon > 0$, we may choose 
 $y = y_{\varepsilon} \geq 3$ so that $ \sum_{p > y} \frac{1}{p(p-1)} < \log(1 + \varepsilon)$. 

 Now, define 
\begin{equation}\label{defineByL}
 \mathcal{B}_{y, \varepsilon} = \{ m \in \mathcal{R}_{y} : G_{y}(m) \geq \log(1 + \varepsilon) \}. 
\end{equation}
 An application of the formulation of Markov's inequality in \eqref{Markovinequality} then gives us that 
\begin{align}
\begin{split}
 \log(1 + \varepsilon) \, \left| \{ m \leq x : m \in \mathcal{B}_{y, \varepsilon} \} \right|
 & \leq \sum_{\substack{ m \leq x \\ m \in \mathcal{R}_{y}}} G_{y}(m) \\ 
 & = S_{y}(x), 
\end{split}\label{splitMarkov}
\end{align}
 recalling \eqref{defineSy}. 
 From \eqref{limfracSy} along with \eqref{splitMarkov}, we find that 
 $$ \overline{d}(\mathcal{B}_{y, \varepsilon}) \leq \frac{\rho_y}{\log(1 + \varepsilon)} \sum_{p > y} \frac{1}{p(p-1)}. $$
 Recalling the density of $\mathcal{R}_{y}$ shown in \eqref{densityRy}, we find that 
 $ \underline{d}\big( \mathcal{R}_{y} \setminus \mathcal{B}_{y, \varepsilon} \big) \geq \rho_y - 
 \overline{d}(\mathcal{B}_{y, \varepsilon})$. 
 Consequently, we obtain the relations
\begin{equation}\label{underlinedlower} 
 \underline{d}\big( \mathcal{R}_{y} \setminus \mathcal{B}_{y, \varepsilon} \big)
 \geq \rho_y \left( 1 - \frac{1}{\log(1 + \varepsilon)} \sum_{p > y} \frac{1}{p(p-1)} \right) > 0. 
\end{equation} 

 Now, if $ m \in \mathcal{R}_{y} \setminus \mathcal{B}_{y, \varepsilon}$, 
 then a combined application of the implication in \eqref{ifminRy}
 together with the complement of the right-hand side of \eqref{defineByL}
 gives us that 
 $\log \frac{\sigma(m)}{m} \leq G_y(m) < \log(1 + \varepsilon)$, 
 i.e., so that 
\begin{equation}\label{imply1plusep} 
 m \in \mathcal{R}_{y} \setminus \mathcal{B}_{y, \varepsilon} 
 \Longrightarrow \frac{\sigma(m)}{m} < 1 + \varepsilon. 
\end{equation}

 So, a combined application of the lower bound in \eqref{underlinedlower} together with 
 the implications in both \eqref{primetosix} 
 \eqref{imply1plusep}
 give us that $\mathcal{R}_{y} \setminus \mathcal{B}_{y, \varepsilon}$
 is a set of positive integers $m$ that is of positive lower density
 such that both $
 (m,6)=1$ and $ \frac{\sigma(m)}{m}<1+\varepsilon$ hold. 
 \end{proof}

\begin{lemma}\label{lem:18m}
Let $m$ be a positive integer with $(m,6)=1$, and put
\begin{equation}\label{definesm}
 s(m)=\sigma(m)-m.
\end{equation}
 If
\begin{equation}\label{1over78}
 \frac{s(m)}{m}<\frac1{78},
\end{equation}
 then $18m$ is (strictly) abundant, but $18m\notin \mathcal{P}$.
\end{lemma}

\begin{proof}
 We henceforth assume that \eqref{1over78} holds. Writing $n = 18 m$, since $m$ and $6$ are coprime, we find that $(18, m) = 1$. 
 Since $\sigma$ is multiplicative, we see that $\sigma(18m) = 39 \sigma(m)$, so that \eqref{definesm} then gives us that $\sigma(18m) 
 = 39(m + s(m))$. Since $m > 0$ and $s(m) \geq 0$, 
 we find that 
 $\sigma(18m) - 2(18m) = 3 m + 39 s(m) > 0$. 
 So, the value $18m$ is (strictly) abundant, with 
 $18 m \in \mathcal{A}$, recalling 
 the definition in \eqref{definecalA}. 
 So, it remains to show that $18 m \not\in \mathcal{P}$. 

 By way of contradiction, suppose that $18 m \in \mathcal{P}$. So, by Lemma \ref{lem:criterion}, the integer
\begin{equation}\label{defineH} 
 H = \frac{\sigma(18m) - 2(18m)}{2}
\end{equation}
 can be written as a sum of distinct divisors of $18m$ apart from $1$ and $18m$. 
 We then rewrite \eqref{defineH} so that 
$$H = \frac{3m + 39 s(m)}{2}. $$
 With the assumption that \eqref{1over78} holds, 
 we find that $39 s(m) < \frac{m}{2}$. 
 So, again with the assumption of 
 \eqref{1over78} holding, we find that 
\begin{equation}\label{chain} 
 \frac{3m}{2} \leq H = \frac{3m + 39 s(m)}{2} < \frac{3m + \frac{m}{2}}{2} = \frac{7m}{4} < 2m.
\end{equation}
 Now, let $\mathscr{S}$ be a sum of proper, nontrivial divisors of $18 m$. 
 Since $(18, m) = 1$, each divisor of $18m$ is of the form $ed$ for $e \mid 18$ and $d \mid m$. 
 We consider the two cases given by $d = m$ and $d < m$. 
 For the $d = m$ case, 
 any such divisors appearing in the expansion of $\mathscr{S}$ as a sum of divisors of $18 m$ 
 provide a contribution of $A m $ to $\mathscr{S}$, for some integer $A \geq 0$. 
 We then write $\mathscr{S} = A m + B$, where
 $$0 \leq B \leq \sum_{e \mid 18} \sum_{\substack{d \mid m \\ d < m}} e d.$$ 
 Since $\sum_{e \mid 18} e = \sigma(18) = 39$
 and since $\sum_{\substack{d \mid m \\ d < m}} d = \sigma(m) - m = s(m)$, 
 we find that $0 \leq B \leq 39 s(m)$. 
 So, the sum $\mathscr{S}$ is of the form 
 $ \mathscr{S} = A m + B $ for a nonnegative integer $A$ and for an integer
 $B \in [0, 39s(m)]$. 
 If $A \leq 1$, then $\mathscr{S} \leq m + 39 s(m)$, 
 and from \eqref{1over78}, we would then have that 
 $\mathscr{S} < \frac{3m}{2}$, so that $\mathscr{S} < H$, recalling 
 \eqref{chain}. Also, if $2 \leq A$, then
 $2m \leq \mathscr{S}$, but \eqref{chain} gives us that $H < 2m$, 
 so that $H < \mathscr{S}$. 
 So, we have shown that no sum of proper, nontrivial divisors of $18 m$ can be equal to 
 the right-hand side of \eqref{defineH}, contradicting Lemma \ref{lem:criterion}. 
 Since we have shown that $18 m \not\in \mathcal{P}$, 
 we have completed our proof. 
\end{proof}

\begin{theorem}\label{thm:main}
 There exists a set $ \mathcal{E} \subseteq \mathcal{A} \setminus \mathcal{P} $ satisfying $\underline d(\mathcal{E}) > 
 0$. Consequently, the strict inequality $ \overline d(\mathcal{P})<d( \mathcal{A} ) $ holds. 
\end{theorem}

\begin{proof}
 By setting $\varepsilon = \frac{1}{78}$ in Lemma \ref{lem:small-ratio}, 
 we find that there exists a set $\mathcal{M}$ of positive integers 
 such that 
\begin{equation}\label{calMlower} 
 \underline{d}(\mathcal{M}) > 0
\end{equation}
 and such that: For each member
 $m$ in $\mathcal{M}$, we have that $(m, 6) = 1$ and that $\frac{\sigma(m)}{m} < \frac{79}{78}$. 
 Again for a member $m$ of $\mathcal{M}$, we set $s(m) = \sigma(m) - m$. 
 We then find that $\frac{s(m)}{m} = \frac{\sigma(m)}{m} - 1 < \frac{1}{78}$. 
 So, if $m \in \mathcal{M}$, then $m$ satisfies all of the required conditions in Lemma \ref{lem:18m}. 
 So, if $m \in \mathcal{M}$, then $18 m$ is strictly abundant and $18 m \not\in \mathcal{P}$. 
 Recalling the definition of $\mathcal{A}$ in \eqref{definecalA}, we thus find that: 
 If $m \in \mathcal{M}$, then
 $18m \in \mathcal{A} \setminus \mathcal{P}$. 
 By then writing $\mathcal{E}$ in place of the family $\{ 18 m : m \in \mathcal{M} \}$, 
 we obtain the containment 
\begin{equation}\label{Econtaineddiff}
 \mathcal{E} \subseteq \mathcal{A} \setminus \mathcal{P}. 
\end{equation}
 Now, observe that 
\begin{equation}\label{xover18}
 \frac{ \left| \{ n \leq x : n \in \mathcal{E} \} \right| }{x} 
 = \frac{1}{18} \frac{ \left| \{ m \leq \frac{x}{18} : m \in \mathcal{M} \} \right| }{x/18}. 
\end{equation}
 Taking lower limits with $x \to \infty$, the relation in \eqref{xover18} then gives us that 
\begin{equation}\label{dEdM} 
 \underline{d}(\mathcal{E}) = \frac{1}{18} \underline{d}(\mathcal{M}). 
\end{equation}
 So, from \eqref{calMlower} and \eqref{dEdM} together, we find that 
\begin{equation}\label{lowerEpos} 
 \underline{d}(\mathcal{E}) > 0. 
\end{equation}

 Now, recall that every $\{ -1, 1\}$-perfect number is a nondeficient number, as observed by Ross \cite{Ross2024} (and recall that this 
 also follows from Lemma \ref{lem:criterion}). So, we find that the containment $\mathcal{P} \subseteq \mathcal{A}$ holds. From this 
 containment together with the containment in \eqref{Econtaineddiff}, we find that: For all $x \geq 1$, the families 
 $\{ n \leq x : n \in \mathcal{P} \}$ and $\{ n \leq x : n \in \mathcal{E} \}$
 are disjoint and are both contained in the family $\{ n \leq x : n \in \mathcal{A} \}$. 
 Consequently, we obtain the relation
\begin{equation}\label{threequotient} 
 \frac{\left| \{ n \leq x : n \in \mathcal{P} \} \right|}{x} 
 + \frac{\left| \{ n \leq x : n \in \mathcal{E} \} \right|}{x} 
 \leq \frac{\left| \{ n \leq x : n \in \mathcal{A} \} \right|}{x}. 
\end{equation}
 From \eqref{threequotient}, we find that 
\begin{equation}\label{ddd}
 \overline{d}(\mathcal{P}) + \underline{d}(\mathcal{E}) \leq d(\mathcal{A}). 
\end{equation}
 From \eqref{lowerEpos} and \eqref{ddd} together, as have, as a consequence, that
 $\overline{d}(\mathcal{P}) < d(\mathcal{A})$, as desired. 
\end{proof}

\section{Disproof of Ross's second conjecture}\label{seconddisproof}
 The following result gives an explicit family of counterexamples
 giving that Ross's second conjecture \cite[Conjecture 17]{Ross2024} does not hold, 
 by letting, for example $a = 945^2$. 

\begin{theorem}
 Let $a$ be an odd abundant square. If $p$ is a prime such that $p > \sigma(a)$, then $n = a p$
 is an odd nonsquare abundant number, but $n \not\in \mathcal{P}$. 
\end{theorem}

\begin{proof}
 Let $a$ denote an odd abundant square. Let $p$ be a prime such that $p > \sigma(a)$. We then write $n$ in place of $ap$. Since 
 $\sigma(a) > 2a$ (from the assumption that $a$ is strictly abundant), we find that $p > \sigma(a) > 2a > a$, 
 with $p \nmid a$. Moreover, since $a \geq 1$, we have that $p > 2$, so that $n$ is odd. 
 Since $p \nmid a$, the exponent of $p$ in the prime factorization of $n = a p$ is $1$.
 So, since $a$ is a square, we see that $n$ is not a square. 

 Now, since $(p, a) = 1$ and since $\sigma$ is multiplicative, 
 we find that $\sigma(n) = \sigma(a) (p+1)$. 
 As a consequence, we have that 
\begin{equation}\label{sigman2n}
 \sigma(n) - 2 n = (\sigma(a) - 2 a) p + \sigma(a). 
\end{equation} 
 From \eqref{sigman2n} together with $a$ being strictly abundant together with the inequality
 $\sigma(a) > 0$, we find that $\sigma(n) - 2 n > 0$, i.e., so that $n$ is abundant. 

 Since $a$ is an odd square, we write 
\begin{equation}\label{FTAa}
 a = \prod_{i=1}^{\omega(a)} p_{a_i}^{2c_i}, 
\end{equation} 
 with $a_i > 1$ and $c_i \geq 1$ for all $i \in \{ 1, 2, \ldots, \omega(a) \}$. 
 From \eqref{basicsigma} and \eqref{FTAa} together, we find that 
\begin{equation}\label{sigmaaprod} 
 \sigma(a) = \prod_{i=1}^{\omega(a)} \big( 1 + p_{a_i} + p_{a_i}^{2} + \cdots + p_{a_{i}}^{2 c_{i}} \big), 
\end{equation}
 with each factor on the right of \eqref{sigmaaprod} being odd, i.e., so that $\sigma(a)$ is odd. 
 So, the integer $\Delta := \sigma(a) - 2 a$ is odd. 
 Also, with the assumption that $a$ is abundant, we find that $\Delta > 0$. 
 So, there exists an integer $q \geq 0$ satisfying $\Delta = 2 q + 1$. 
 As a consequence of \eqref{sigman2n}, we obtain that 
 $ \sigma(n) - 2 n = \Delta p + \sigma(a)$, and we define
 $ H(n) := \frac{\sigma(n) - 2 n}{2}. $ From the definition, we find that 
\begin{align*} 
 H(n) 
 & = \frac{\Delta p + \sigma(a)}{2} \\ 
 & = q p + \frac{p + \sigma(a)}{2}.
\end{align*}

 Now, by way of contradiction, suppose that $n \in \mathcal{P}$. So, by Lemma \ref{lem:criterion}, the integer $H(n)$ can be as a sum of 
 proper, nontrivial divisors of $n$. Now, let $\mathscr{S}$ denote a sum of nontrivial and proper divisors of $n$. Since $n = ap$ and since 
 $a$ and $p$ are relatively prime, each divisor of $n$ is either a divisor $d$ of $a$ or is of the form $dp$, again for a divisor $d$ of $a$. 
 This gives us that $\mathscr{S} = A p + B$ for some $A \geq 0$ and an integer $B$, 
 which satisfies the bounds whereby $0 \leq B \leq \sigma(a)$. 
 From the condition $\sigma(a) < p$, 
 we find that 
\begin{equation}\label{qpsigma} 
 q p + \sigma(a) < q p + \frac{p + \sigma(a)}{2} = H(n).
\end{equation}
 In a similar fashion, we have that 
\begin{equation}\label{qpfrac} 
 H(n) = q p + \frac{p + \sigma(a)}{2} < q p + p. 
\end{equation}
 If $A \leq q$, then $Ap + B \leq q p + \sigma(a) < H(n)$, from \eqref{qpsigma}. If $A \geq q + 1$, then $A p + B \geq (q+1) p > H(n)$, 
 from \eqref{qpfrac}. So, no divisor sum of the form $\mathscr{S}$ can be equal to $H(n)$, contradicting (via Lemma \ref{lem:criterion}) 
 our assumption that $n \in \mathcal{P}$. 

 So, we have shown that $n$ is odd, nonsquare, and abundant, with $n \not\in \mathcal{P}$. 
\end{proof}

\subsection*{Acknowledgements}
 The author thanks Joshua Zelinsky for some useful feedback, and the author acknowledges extensive interactions with GPT-5.5 
 Pro during the exploratory and proof-development stages of this work. All AI-generated suggestions were 
 substantially revised, corrected, and independently verified by the author, who assumes full responsibility for the mathematical content.

\end{document}